\documentclass[onecolumn]{article}
\usepackage[margin=1in]{geometry}
\usepackage[dvipdfmx]{graphicx,color}
\usepackage{CJKutf8}
\usepackage{amsthm}
\usepackage{amsfonts, amssymb, bm}
\usepackage{mathtools}
\usepackage{tcolorbox}
\usepackage{subfig}
\usepackage{ascmac}
\usepackage{graphicx}
\usepackage{comment}
\usepackage{mathrsfs}
\usepackage{amscd}
\usepackage{url}
\usepackage[all]{xy}
\usepackage{bbm}
\usepackage{tikz-cd}
\usepackage{array}
\usepackage{amsmath}
\usepackage{tikz}
\usepackage{cancel}
\usepackage{color}
\usepackage{here}
\usepackage{listings}
\usepackage{float}
\usepackage{verbatim}
\usepackage{lipsum} 
\usepackage{multicol} 
\usepackage{xcolor}
\usepackage{colortbl}
\usepackage{longtable}
\usepackage{multirow}
\usepackage{amssymb}
\usepackage{booktabs}
\usepackage{pdflscape}
\usepackage{graphicx} 
\usepackage{empheq}
\usepackage{hyperref}
\usepackage{cases}
\usepackage{enumitem}
\usepackage{setspace}
\hypersetup{hidelinks}
\numberwithin{equation}{section}

\usepackage[dvipdfmx]{graphicx}
\usetikzlibrary{calc}
\tcbuselibrary{breakable}

\makeatletter
\newcommand{\address}[1]{\gdef\@address{#1}}

\newcommand{\printaddress}{\@address}
\makeatother
\newtheoremstyle{mytheoremstyle}
  {10pt}
  {10pt}
  {\itshape}
  {}
  {\bfseries}
  {}
  {1em}
{\thmname{#1}\thmnumber{ #2}\thmnote{ \textup{#3}}}

\theoremstyle{mytheoremstyle}

\newtheorem{thm}{Theorem}[section]
\newtheorem{prop}[thm]{Proposition}
\newtheorem{defi}[thm]{Definition}
\newtheorem{fac}[thm]{Fact}
\newtheorem{exam}[thm]{Example}
\newtheorem{remark}[thm]{Remark}
\newtheorem{lem}[thm]{Lemma}
\newtheorem{cor}[thm]{Corollary}

\makeatletter
\renewenvironment{proof}[1][\proofname]{%
  \par\normalfont 
  \topsep6pt\relax 
  \trivlist\item[\hskip\labelsep\itshape #1\@addpunct{.}]\ignorespaces
}{%
  \qed\endtrivlist\@endpefalse
}
\makeatother

\DeclareMathOperator{\PSL}{\mathrm{PSL}}

\DeclareMathOperator{\Aut}{\mathrm{Aut}}

\DeclareMathOperator{\Z}{\mathbb{Z}}
\DeclareMathOperator{\id}{\mathrm{id}}

\DeclareMathOperator{\SD}{\mathrm{SD}}

\makeatletter
\@ifundefined{Ker}{\DeclareMathOperator{\Ker}{Ker}}{}
\@ifundefined{End}{\DeclareMathOperator{\End}{End}}{}
\@ifundefined{Aut}{\DeclareMathOperator{\Aut}{Aut}}{}
\makeatother



\makeatletter

\renewcommand*\@cite[2]{%
  (#1\if\relax\detokenize{#2}\relax\else , #2\fi)%
}

\renewcommand*\@biblabel[1]{(#1)}

\makeatother

\makeatletter
\newcommand{\citenp}{\@ifnextchar[{\@citenp}{\@citenp[]}}
\newcommand{\@citenp}[2][]{%
  \begingroup
  \renewcommand*\@cite[2]{%
    ##1\if\relax\detokenize{##2}\relax\else , ##2\fi%
  }%
  \cite[#1]{#2}%
  \endgroup
}
\makeatother

\begin{document}
\doublespacing

\title{Non-isomorphic metacyclic 
$p$-groups of split type with the same group zeta function}

\author{
Yuto Nogata\\
\normalsize{Graduate School of Science and Technology, Hirosaki University}\\
\normalsize{3 Bunkyo-cho, Hirosaki, Aomori, 036-8560, Japan}\\
\normalsize{Email: h24ms113@hirosaki-u.ac.jp}
}
\date{\today}

\maketitle

\section*{Abstract}
For a finite group $G$, let $a_n(G)$ be the number of subgroups of order $n$ and define
$\zeta_G(s)=\sum_{n\ge 1} a_n(G)n^{-s}$.
Examples are known of non-isomorphic finite groups with the same group zeta function.
However, no general criterion is known for when two finite groups have the same group zeta function.
Fix integers $m,n\ge 1$ and a prime $p$, and consider the metacyclic $p$-groups of split type $G(p,m,n,k)$ defined by $
G(p,m,n,k)=\langle a,b \mid a^{p^{m}}=b^{p^{n}}=\id, b^{-1}ab=a^{k}\rangle$.
For fixed $m$ and $n$, we characterize the pairs of parameters $k_1,k_2$ for which
$\zeta_{G(p,m,n,k_1)}(s)=\zeta_{G(p,m,n,k_2)}(s)$.

\section{Introduction}\label{section:meta-introduction-start}
There is a long tradition of studying subgroup growth and related subgroup counting problems.
Classical congruence results for the number of subgroups of index $p$ in finite $p$-groups go back to \cite{Miller1919} and were strengthened in \cite{Dyubyuk1950}.
For free groups, a recursion for the number of subgroups of index $n$ is given in \cite{Hall1949}.
For the modular group $\PSL(2,\Z)$, explicit computations and asymptotics for the number of subgroups of index $n$
are studied in \cite{Imrich1978}, \cite{Stothers1977}.

Motivated by these developments, the Dirichlet series enumerating finite-index subgroups is defined by
\[
  \zeta^{\mathrm{idx}}_{G}(s)
    :=\sum_{\substack{H\le G\\ [G:H]<\infty}} [G:H]^{-s}
    =\sum_{n\ge 1}\frac{a^{\mathrm{idx}}_{n}(G)}{n^{s}},
  \qquad
  a^{\mathrm{idx}}_{n}(G):=\#\{\,H\le G\mid [G:H]=n\,\}.
\]
It was introduced in \cite{Smith1983}, and it was studied systematically for nilpotent groups in \cite{GSS1988}. Analogous Dirichlet series enumerating normal finite-index subgroups are studied in \cite{Lee2022}.

In this paper, we consider the group zeta function that enumerates all subgroups of a finite group by their orders.
Following \cite{Hironaka2017}, we denote by $S(G)$ the set of all subgroups of a finite group $G$.
For each integer $n\ge 1$, we define $a_n(G):=\#\{\,H\in S(G)\mid |H|=n\,\}$.
We define the group zeta function of $G$ by
\[
  \zeta_G(s)=\sum_{H\in S(G)}|H|^{-s}
  =\sum_{n\ge 1}\frac{a_n(G)}{n^{s}}.
\]
For isomorphic finite groups $G\simeq G'$, we trivially have $\zeta_G(s)=\zeta_{G'}(s)$.
In general, no criterion is known for when non-isomorphic groups can have the same group zeta function.
In \cite{Hironaka2017}, the following pair is shown to provide an example of this phenomenon:
\begin{align*}
G_\lambda
  &= \langle a_1,a_2,\dots,a_n \mid |a_i| = p^{\lambda_i},\ a_i a_j = a_j a_i \rangle
     \simeq \prod_{i=1}^{n} \mathbb{Z}_{p^{\lambda_i}},\\
\widetilde{G}_\lambda
  &= \langle a_1,a_2,\dots,a_n \mid |a_i| = p^{\lambda_i},\
       a_n a_1 a_n^{-1} = a_1^{\,1+p^{\lambda_1-1}},\
       a_i a_j = a_j a_i \text{ unless } (i,j)=(1,n) \rangle
  \nonumber\\
  &\simeq \prod_{i=2}^{n-1} \mathbb{Z}_{p^{\lambda_i}}
        \times \bigl(\mathbb{Z}_{p^{\lambda_1}}\rtimes
                      \mathbb{Z}_{p^{\lambda_n}}\bigr).
\end{align*}
Outside $p$-groups, further examples of non-isomorphic finite groups with the same group zeta function
have been found, for instance in \cite{ParkerKanchana2014}.
More generally, no necessary and sufficient condition for this phenomenon is currently known.

A finite group $G$ is called \emph{metacyclic} if $G$ has a cyclic normal subgroup
$N\lhd G$ and the quotient $G/N$ is also cyclic.
Standard presentations of finite metacyclic groups were given in \cite{Hempel2000}.
For arbitrary finite $p$-groups, \cite{Wall1961} studied general relations and inequalities
between the numbers of subgroups of each order and each index, using Eulerian functions.
For metacyclic $p$-groups, subgroup-counting formulas were obtained in \cite{Berkovich2011}, \cite{Mann2010}.
In particular, these results imply that for an odd prime $p$ and a fixed integer $\ell\ge 1$,
any two metacyclic $p$-groups of order $p^{\ell}$ with the same exponent have the same group zeta function.
Thus, for metacyclic $p$-groups with $p\ge 3$, the problem posed in \cite{Hironaka2017}
admits a partial answer.

For $p=2$, the subgroup-count formulas in \cite{Berkovich2011} are given by case distinctions.
As summarized in Fact~\ref{fac:metacyclic-2-subgroup-numbers}, these distinctions depend on the structure of certain quotient groups.
They determine the numbers $a_{2^{t}}(G)$ for each individual metacyclic $2$-group.
On the other hand, they do not explicitly describe which metacyclic $2$-groups are non-isomorphic
but have the same group zeta function.

In this paper, we focus primarily on the case $p=2$.
For a prime $p$ and integers $m,n\ge 1$, we denote by $G(p,m,n,k)$ the group defined in \eqref{eq:Gpmnk} below.
By \cite{Hempel2000}, every finite metacyclic $p$-group of split type is isomorphic to $G(p,m,n,k)$
for some integers $m,n\ge 1$ and some "valid" $k$.
For $p=2$, we give a necessary and sufficient condition for two non-isomorphic groups in this family
to have the same group zeta function.
The condition depends on the parameters $m$ and $n$, and it is expressed in terms of the $2$-adic valuations
$v_{2}(k+1)$, where $v_{2}(\,\cdot\,)$ denotes the $2$-adic valuation.

For $p\ge 3$, our framework also shows that any two metacyclic $p$-groups of split type
of order $p^{m+n}$ with the same exponent have the same group zeta function.

\begin{thm}[Main Theorem]\label{thm:metamaintheorem}
Let $p$ be a prime and let $m,n\ge1$ be integers.
Set $H:=\mathbb{Z}_{p^{m}}=\langle a\rangle$ and
$K:=\mathbb{Z}_{p^{n}}=\langle b\rangle$, and define
\begin{equation}\label{eq:Gpmnk}
G(p,m,n,k)
  :=\langle a,b \mid a^{p^{m}}=b^{p^{n}}=\id,\ bab^{-1}=a^{k}\rangle.
\end{equation}
We say that $G(p,m,n,k)$ is "\emph{valid}" if this presentation yields a finite
metacyclic $p$-group of split type.
In this case, $G(p,m,n,k)=H\rtimes_{\phi}K$, where $\phi\colon K\to \Aut(H)$ is the homomorphism determined by
$\phi(b)(a)=bab^{-1}=a^{k}$.
Then the following hold.
\begin{enumerate}
 \item[(I)] the group $G(p,m,n,k)$ is valid if and only if
 \begin{equation}\label{eq:welldefi条件intro}
  k^{p^{n}}\equiv 1 \pmod{p^{m}}.
\end{equation}

 \item[(II)] Assume that $k_1$ and $k_2$ satisfy \eqref{eq:welldefi条件intro}. Then the following equivalence holds:
\[
   G(p,m,n,k_{1})\simeq G(p,m,n,k_{2})
  \iff
  \exists v\in (\mathbb{Z}/p^{n}\mathbb{Z})^{\times}
  \text{ such that } k_{2}\equiv k_{1}^{v} \pmod{p^{m}}.
\]

 \item[(III)] Assume that $k_1$ and $k_2$ satisfy \eqref{eq:welldefi条件intro}. A necessary and sufficient condition for
 $\zeta_{G(p,m,n,k_{1})}(s)=\zeta_{G(p,m,n,k_{2})}(s)$ to hold
 is given as follows:
 \begin{itemize}
 \item[(1)] If $p$ is an odd prime, then for any $k_1,k_2$
 satisfying \eqref{eq:welldefi条件intro}, we have
 $\zeta_{G(p,m,n,k_{1})}(s)=\zeta_{G(p,m,n,k_{2})}(s)$.

 \item[(2)] If $p=2$ and $n<m$, then for any $k_1,k_2$
 satisfying \eqref{eq:welldefi条件intro}, we have
 $\zeta_{G(p,m,n,k_{1})}(s)=\zeta_{G(p,m,n,k_{2})}(s)$
 if and only if one of the following conditions holds:
 \[
 \begin{cases}
 \text{\textup{(A)}}\ \ v_2(k_1+1)=v_2(k_2+1)\le m-n,\\
 \text{\textup{(B)}}\ \ \min\{v_2(k_1+1), v_2(k_2+1)\}> m-n.
 \end{cases}
 \]

 \item[(3)] If $p=2$ and $n\ge m$, then for any $k_1,k_2$
 satisfying \eqref{eq:welldefi条件intro}, we have
 $\zeta_{G(p,m,n,k_{1})}(s)=\zeta_{G(p,m,n,k_{2})}(s)$.
 \end{itemize}
\end{enumerate}
In particular, by choosing representatives $k_{i}$ of the isomorphism classes
as in {\rm(II)} and applying the criterion in {\rm(III)}, we can systematically
obtain examples of non-isomorphic metacyclic $p$-groups of split type whose zeta
functions coincide.
\end{thm}

\section{Notation and Preliminaries}\label{section:meta-Notation and Preliminary-start}

\begin{defi}
Throughout this paper, all groups are assumed to be finite.
We use the following notation and terminology.
\begin{itemize}
  \item $S(G)$: the set of all subgroups of a group $G$,
  \item $a_{n}(G) := \#\{\,H\in S(G)\mid |H|=n\,\}$,
  \item $p$: a prime,
  \item $\zeta_{G}(s) := \sum_{n=1}^{\infty} \frac{a_{n}(G)}{n^{s}}$,
  \item $H := \langle a \rangle \simeq \mathbb{Z}_{p^{m}}$,
        $K := \langle b\rangle \simeq \mathbb{Z}_{p^{n}}$,
  \item $G(p,m,n,\lambda,k) :=
  \langle a,b \mid
    a^{p^{m}}=\id,\ 
    b^{p^{n}}=a^{p^{\lambda}},\ 
    bab^{-1} = a^{k}
  \rangle$,
  \item $G(p,m,n,k)
  :=\langle a,b \mid a^{p^{m}}=b^{p^{n}}=\id,\ bab^{-1}=a^{k}\rangle$,
  \item A group $L$ is called a \emph{section} of a group $G$ if there exist
        a subgroup $A\subset G$ and a normal subgroup $B\lhd A$ such that
        $L\simeq A/B$,
  \item For a prime $p$ and $d\in \Z$, we denote by $v_p(d)$ the $p$-adic valuation of $d$, and we adopt the convention $v_p(0):=\infty$,
  \item For a finite $p$-group $P$, we define $\mho_{1}(P):=\langle\,x^{p} \mid x\in P\,\rangle$,
  \item For a finite group $G$, the \emph{exponent} of $G$ is
$\exp(G):=\mathrm{lcm}\{\,|g|\mid g\in G\,\}$.
\end{itemize}
\end{defi}

By \cite{Hempel2000}, every finite metacyclic $p$-group is isomorphic to a group
given by a presentation of the form $G(p,m,n,\lambda,k)$ for suitable parameters.
Moreover, a finite metacyclic $p$-group is split if and only if it is isomorphic
to one given by a presentation of the form $G(p,m,n,k)$.
Therefore, it suffices to work with these presentations.
We record the necessary and sufficient congruence conditions under which the above presentations
define finite metacyclic $p$-groups.

We say that $G(p,m,n,\lambda,k)$ is \emph{valid} if this presentation yields a finite metacyclic $p$-group.
We say that $G(p,m,n,k)$ is \emph{valid} if this presentation yields a finite metacyclic $p$-group of split type.

\begin{fac}[\cite{Hempel2000}]
The group $G(p,m,n,\lambda,k)$ is valid if and only if
\[
k^{p^{n}} \equiv 1 \pmod{p^{m}}
\quad\text{and}\quad
p^{\lambda}(k-1) \equiv 0 \pmod{p^{m}}.
\]
Moreover, the group $G(p,m,n,k)$ is valid if and only if
\begin{equation}\label{eq:well-definedness}
k^{p^{n}} \equiv 1 \pmod{p^{m}}.
\end{equation}
\end{fac}

In this paper, we consider only valid groups of split type of the form $G(p,m,n,k)$.

\begin{fac}[\cite{HumphriesSkabelund2015}]\label{fac:split-meta-iso}
Let $\alpha,\beta\ge 1$ and let $ G_i = \langle a,b \mid a^{\alpha}=\id,\ b^{\beta}=\id,\ bab^{-1} = a^{\gamma_i}\rangle$ for $i=1,2$. Then the following equivalence holds:
\[
  G_1 \simeq G_2
  \quad\Longleftrightarrow\quad
  \gamma_1,\gamma_2\in (\mathbb{Z}/\alpha\mathbb{Z})^{\times}
  \text{ and }
  \langle \gamma_1\rangle = \langle \gamma_2\rangle
  \text{ as subgroups of } (\mathbb{Z}/\alpha\mathbb{Z})^{\times}.
\]

\end{fac}

\begin{cor}\label{cor:isom-Gpmn}
Let $m,n\ge 1$, and let $k_{1},k_{2}$ be integers satisfying
\eqref{eq:well-definedness}.
Then $G(p,m,n,k_{1})\simeq G(p,m,n,k_{2})$ holds if and only if
\begin{equation}\label{eq:同型の条件メタサイクル}
  \exists v\in (\mathbb{Z}/p^{n}\mathbb{Z})^{\times}
  \textup{ such that }
  k_{2}\equiv k_{1}^{v} \pmod{p^{m}}.
\end{equation}
\end{cor}

\begin{defi}\label{def:quasi-regular}
Let $p$ be a prime and let $G$ be a finite metacyclic $p$-group.
We say that $G$ is \emph{quasi-regular} if the following conditions are satisfied:
\begin{description}[labelwidth=0.7cm]
  \item[\textup{(QR1)}]
    If $p=2$, then $G$ has no non-abelian section of order $8$.
  \item[\textup{(QR2)}]
    For every section $H$ of $G$, the set $\{\,x^{p}\mid x\in H\,\}$ is a subgroup of $H$,
    and it coincides with $\mho_{1}(H)$.
\end{description}
\end{defi}

\begin{fac}[\cite{Mann2010}]\label{fac:メタサイクル群の準正則性}
Every finite metacyclic $p$-group with $p>2$ is quasi-regular.
\end{fac}

\begin{fac}[\cite{Berkovich2011}]\label{fac:準正則な部分群個数}
Let $G$ be a finite quasi-regular metacyclic $p$-group of order $p^{\ell}$.
Assume that $\exp(G)=p^{e}<p^{\ell}$, and let $0\le t\le \ell$.
Set $f:=\ell-e\ge 1$.
Then the number $a_{p^{t}}(G)$ is given by
\[
  a_{p^{t}}(G)
  =
  \begin{cases}
    \dfrac{p^{t+1}-1}{p-1} & \text{if } t\le f,\\
    \dfrac{p^{f+1}-1}{p-1} & \text{if } f<t\le e,\\
    \dfrac{p^{\,\ell-t+1}-1}{p-1} & \text{if } t>e.
  \end{cases}
\]
\end{fac}

\begin{cor}\label{cor:pge3-all-similar}
Let $p\ge 3$ and fix integers $m,n\ge 1$.
Then for any integers $k_{1},k_{2}$ satisfying
\eqref{eq:well-definedness}, we have
\[
  \zeta_{G(p,m,n,k_{1})}(s)
  =
  \zeta_{G(p,m,n,k_{2})}(s).
\]
\end{cor}

\begin{proof}
For fixed $p,m,n$ and any integer $k$ satisfying \eqref{eq:well-definedness},
the order and exponent of $G(p,m,n,k)$ are given by
\[
  |G(p,m,n,k)|=p^{m+n},\qquad
  \exp\bigl(G(p,m,n,k)\bigr)=p^{\max\{m,n\}}<p^{m+n}.
\]
By Fact~\ref{fac:メタサイクル群の準正則性}, each $G(p,m,n,k)$ is a finite
quasi-regular metacyclic $p$-group.
Applying Fact~\ref{fac:準正則な部分群個数}, we see that for finite quasi-regular
metacyclic $p$-groups with the same order $|G|$ and exponent $\exp(G)$ the
numbers $a_{p^{t}}(G)$ coincide for all $t$.
Hence, for any integers $k_{1},k_{2}$ satisfying \eqref{eq:well-definedness}, we obtain, for every integer $t\ge 0$,
\[
  a_{p^{t}}\bigl(G(p,m,n,k_{1})\bigr)
  =
  a_{p^{t}}\bigl(G(p,m,n,k_{2})\bigr).
\]
Therefore, $\zeta_{G(p,m,n,k_{1})}(s)=\zeta_{G(p,m,n,k_{2})}(s)$.
\end{proof}

Thus, for $p\ge 3$, the group zeta function is the same for all groups $G(p,m,n,k)$
with $k$ satisfying \eqref{eq:well-definedness}.
For $p=2$, the following is known.

\begin{fac}[\cite{Berkovich2011}]\label{fac:metacyclic-2-RG-quotient}
Let $G$ be a finite metacyclic $2$-group.
For each integer $i\ge0$, define
\[
  \Omega_i(G) := \{\,x\in G \mid x^{2^i}=\id\,\},\qquad
  w:=\max\{\, i\ge 0 \mid |\Omega_i(G)| = 2^{2i} \,\},\qquad
  R:=\Omega_{w}(G).
\]
Then $|R| = 2^{2w}$, and the quotient $G/R$ is isomorphic to one of
the following:
\begin{itemize}
  \item $\{\id\}$,
  \item $\mathbb{Z}_{2^{c}}$ for some $c\ge1$,
  \item $D_{2^{c}}, Q_{2^{c}}$, or
$\SD_{2^{c}}$ for some $c\ge3$.
\end{itemize}

\end{fac}

\begin{fac}[\cite{Berkovich2011}]\label{fac:metacyclic-2-subgroup-numbers}
Let $G$ be a finite metacyclic $2$-group of order $2^{\ell}$.
For each integer $t$ with $0\le t\le \ell$, the value of $a_{2^{t}}(G)$
is determined by the structure of $G$ as follows.
\begin{enumerate}
  \item[\textup{(1)}]
  If $G\simeq \mathbb{Z}_{2^{\ell}}$, then $a_{2^{t}}(G)=1$ for every $0\le t\le \ell$.

\item[\textup{(2)}]
Assume that $w=0$ and that $G \simeq D_{2^{\ell}}, Q_{2^{\ell}},$ or $\SD_{2^{\ell}}$.
Then $a_{1}(G)=a_{2^{\ell}}(G)=1$, and for every $2\le t<\ell$ we have
\[
  a_{2^{t}}(G)=2^{\,\ell-t}+1.
\]
Moreover, for $t=1$ we have
\[a_{2}(D_{2^{\ell}})=2^{\,\ell-1}+1,\qquad a_{2}(Q_{2^{\ell}})=1,\qquad a_{2}(\SD_{2^{\ell}})=2^{\ell-2}+1  \]

  \item[\textup{(3)}]
  Assume that $w>0$ and $R=G$.
  Then $|G|=2^{2w}$, and for every $0\le t\le 2w$,
  \[
    a_{2^{t}}(G)
    =
    \begin{cases}
      2^{t+1}-1         & \text{if } 0\le t\le w,\\
      2^{\,2w-t+1}-1    & \text{if } w<t\le 2w.
    \end{cases}
  \]

  \item[\textup{(4)}]
  Assume that $w>0$, $R\subsetneq G$, and $G/R\simeq \mathbb{Z}_{2^{c}}$
  for some $c\ge1$.
  Then $|G|=2^{2w+c}$, and for every $0\le t\le 2w+c$,
  \[
    a_{2^{t}}(G)
    =
    \begin{cases}
      2^{t+1}-1        & \text{if } 0\le t\le w,\\
      2^{\,w+1}-1      & \text{if } w<t<w+c,\\
      2^{\,2w+c-t+1}-1 & \text{if } w+c\le t\le 2w+c.
    \end{cases}
  \]

  \item[\textup{(5)}]
  Assume that $w>0$, $R\subsetneq G$, and
  $G/R\simeq D_{2^{c}}, Q_{2^{c}}$ or $\SD_{2^{c}}$ for some $c\ge3$.
  Then $|G|=2^{2w+c}$.
  The numbers $a_{2^{t}}(G)$ for $0\le t\le 2w+c$ are given by explicit formulas in
  \cite{Berkovich2011}.
  We omit the full list of formulas here because it is lengthy.
  \begin{itemize}
    \item If $G/R\simeq Q_{2^{c}}$, see \cite{Berkovich2011}, Theorem~3.9 and Supplements~1--2
    \item If $G/R\simeq D_{2^{c}}$, see \cite{Berkovich2011}, Theorem~3.11 and Supplements~1--3
    \item If $G/R\simeq \SD_{2^{c}}$, see \cite{Berkovich2011}, Theorem~3.12 and Supplements~1--2.
  \end{itemize}

  \end{enumerate}
\end{fac}

Fact~\ref{fac:metacyclic-2-subgroup-numbers} determines the values of $a_{2^{t}}(G)$.
However, counting pairs of non-isomorphic groups with coincident group zeta functions
requires one to keep track of the data $w$, $R$, and the quotient group $G/R$.
As a result, the number of such pairs is hard to determine even when $|G|$ is fixed.
The difficulty becomes more pronounced as the order grows.

In contrast, for metacyclic $p$-groups of split type of the form $G(p,m,n,k)$,
zeta equality admits a uniform characterization.
When $p=2$, Theorem~\ref{thm:metamaintheorem} gives a criterion for zeta equality in this family
in terms of $m$, $n$, and the $2$-adic valuations $v_{2}(k+1)$.
For odd primes $p$, the same method yields an analogous criterion.
This criterion is parallel to Corollary~\ref{cor:pge3-all-similar}.
Therefore, we focus primarily on the case $p=2$.

We will frequently use Lemma~\ref{lem:LTElemma} in our computations.
\begin{lem}[\cite{AndreescuReflections2011}]\label{lem:LTElemma}
Let $p$ be a prime, $x,y\in\mathbb{Z}$, and $n\ge 1$.

If $p$ is odd, $p\mid(x-y)$, $p\nmid x$, and $p\nmid y$, then
\begin{equation}\label{eq:LTE-podd}
  v_p(x^n-y^n)=v_p(x-y)+v_p(n).
\end{equation}

If $p=2$ and $x\equiv y\equiv 1\pmod 2$, then
\begin{subequations}\label{eq:LTE-p2}
\begin{numcases}{v_2(x^n-y^n)=}
  v_2(x-y), & if $n$ is odd, \label{eq:LTE-peven-nodd}\\
  v_2(x-y)+v_2(x+y)+v_2(n)-1, & if $n$ is even. \label{eq:LTE-peven-neven}
\end{numcases}
\end{subequations}
\end{lem}

\section{Criteria for zeta equality among the groups $G(p,m,n,k)$}

Any semidirect product $G:=V\rtimes_{\phi}U$ fits into a split short exact sequence
\[
  1\to V\to G\overset{\pi}{\to}U\to 1.
\]
For subgroups $W\subset V$ and $L\subset U$, we define
\[
  S_{W,L}(G):=\{\,H\le G\mid H\cap V=W,\ \pi(H)=L\,\}.
\]
For every subgroup $H\subset G$, set $W:=H\cap V$ and $L:=\pi(H)$.
Then there is a short exact sequence
\[
  1\to W\to H\overset{\pi\mid_{H}}{\to}L\to 1.
\]

Let $W\subset V$ and $L\subset U$.
Then $W\lhd \pi^{-1}(L)$ if and only if $W\lhd V$ and
$\phi(\ell)(W)=W$ for every $\ell\in L$.
Assume that $W\lhd \pi^{-1}(L)$.
Set $\tilde{V}:=V/W$ and $\tilde{G}:=\pi^{-1}(L)/W$.
Then factoring out $W$ yields a short exact sequence
\[
  1\to \tilde{V}\to \tilde{G}\overset{\tilde{\pi}}{\to}L\to 1.
\]
The group $L$ acts on $\tilde{V}$ via conjugation in $\tilde{G}$.
According to \cite{Robinson1996}, there is a bijection
\[
  S_{\{\id\},L}(\tilde{G})
  :=\{\,\tilde{H}\le \tilde{G}\mid \tilde{H}\cap \tilde{V}=\{\id\},\
      \tilde{\pi}(\tilde{H})=L\,\}
  \longleftrightarrow Z^{1}(L,\tilde{V}).
\]
Moreover, the map $H\mapsto H/W$ induces a bijection
$S_{W,L}(G)\longleftrightarrow S_{\{\id\},L}(\tilde{G})$.

We now apply this to our situation.
From now on, we write $G:=G(p,m,n,k)$.
For every subgroup $A\subset H\simeq \mathbb{Z}_{p^{m}}$, the subgroup $A$ is
characteristic in $H$.
Since $H\lhd G$, we have $A\lhd G$.
Therefore, for every subgroup $B\subset K$, the subgroup $A$ is $B$-invariant.
We obtain a bijection
\begin{equation}\label{eq:semid-Z1-1}
S_{A,B}(G):=\Bigl\{\,T\le G\ \Bigm|\ T\cap H=A,\ \pi(T)=B\,\Bigr\}
\longleftrightarrow Z^{1}\bigl(B,\,H/A\bigr).
\end{equation}
In particular, we have the disjoint union $S(G)=\bigsqcup_{A\subset H,\, B\subset K} S_{A,B}(G)$.
For every subgroup $T\subset G$, the pair $(A,B)$ is uniquely determined by
$A=T\cap H$ and $B=\pi(T)$.
Combining this with \eqref{eq:semid-Z1-1}, we obtain
\[
  \#S(G)=\sum_{A\le H,\, B\le K} \bigl|Z^{1}(B,H/A)\bigr|.
\]

For each $0\le i\le m$ and $0\le j\le n$, let $A_i\le H$ and $B_j\le K$ be the unique subgroups
with $|A_{i}|=p^{i}$ and $|B_{j}|=p^{j}$.
Restricting to subgroups of order $p^{t}$, we obtain
\begin{equation}\label{eq:apt-Z1-1}
a_{p^{t}}(G)=\sum_{\substack{0\le i\le m,\ 0\le j\le n\\ i+j=t}}
\Bigl|\,Z^{1}\bigl(B_{j},\,H/A_{i}\bigr)\Bigr|.
\end{equation}
Putting $j=t-i$ and using $0\le j\le n$, we obtain $t-n\le i\le t$.
Hence $\max\{0,\,t-n\}\le i\le \min\{t,\,m\}$.
Thus, we can rewrite \eqref{eq:apt-Z1-1} as
\begin{equation}\label{eq:apt-Z1-2}
a_{p^{t}}(G)
=
\sum_{i=\max\{0,\,t-n\}}^{\min\{t,\,m\}}
\Bigl|\,Z^{1}\bigl(B_{\,t-i},\,H/A_{i}\bigr)\Bigr|.
\end{equation}
For each $0\le i\le m$ and $0\le j\le n$ with $t=i+j$, we define
\[
A_i:=\langle a^{p^{\,m-i}}\rangle \simeq \mathbb{Z}_{p^{i}} \subset H,\quad
B_j:=\langle b^{p^{\,n-j}}\rangle \simeq \mathbb{Z}_{p^{j}} \subset K,\quad
M_i:=H/A_i\ \simeq\ \mathbb{Z}_{p^{\,m-i}}.
\]

\begin{fac}[c.f.\cite{Robinson1996}, Chap.~3]\label{fac:Z1-kernel-Nij}
Identify $M_i$ with $\mathbb{Z}/p^{\,m-i}\mathbb{Z}$, and write it additively.
Let $T_{i,j}\in\End(M_i)$ be the endomorphism induced by the action of $B_j$ on $M_i$
via the restriction of $\phi$.
Then
\begin{equation}\label{eq:Tijformula}
T_{i,j}(x)=u_{i,j}\,x\quad(x\in\mathbb{Z}/p^{\,m-i}\mathbb{Z}),\quad
u_{i,j}\equiv k^{\,p^{\,n-j}}\pmod{p^{\,m-i}}.
\end{equation}
Define
\begin{equation}\label{eq:Nijdefi}
N_{i,j}\ :=\ 1+T_{i,j}+T_{i,j}^{\,2}+\cdots+T_{i,j}^{\,p^{\,j}-1},
\end{equation}
where $1=\id_{M_{i}}$.
With this definition, there is a bijection
\begin{equation}\label{eq:Z1-as-kernel}
Z^{1}(B_j,M_i)\ \longleftrightarrow\ \Ker(N_{i,j})
=\{\,x\in M_i\mid N_{i,j}(x)=0\,\}.
\end{equation}
\end{fac}

\begin{cor}\label{cor:aptG-KerNij-eq}
Combining \eqref{eq:apt-Z1-2}
with \eqref{eq:Z1-as-kernel}, we obtain
\[
a_{p^{t}}(G)
=
\sum_{i=\max\{0,\,t-n\}}^{\min\{t,\,m\}}
\Bigl|\Ker\bigl(N_{i,\,t-i}\bigr)\Bigr|.\]
\end{cor}

\begin{prop}\label{prop:KerNij-podd-2-case}
For $N_{i,j}$ defined in \eqref{eq:Nijdefi}, the order $|\Ker(N_{i,j})|$ is given as follows.

Assume that $p$ is odd. Then
\[
  \bigl|\Ker(N_{i,j})\bigr|
  \ =\ p^{\,\min\{\,m-i,\ j\,\}}.
\]

Assume that $p=2$. Then
\[
  \bigl|\Ker(N_{i,j})\bigr|
  \ =\
  \begin{cases}
    2^{\,\min\{\,m-i,\ j\,\}} &
      \textup{if } u_{i,j}\equiv 1\pmod{2^{\,m-i}},\\
    2^{\,\min\{\,m-i,\ v_2(u_{i,j}+1)+j-1\,\}} &
      \textup{if } u_{i,j}\not\equiv 1\pmod{2^{\,m-i}}.
  \end{cases}
\]
\end{prop}

\begin{proof}
Put $r:=m-i$.
Under the identification $M_i\simeq\mathbb{Z}/p^{r}\mathbb{Z}$, the formulas
\eqref{eq:Tijformula} and \eqref{eq:Nijdefi} yield
\[
  T_{i,j}(x)=u_{i,j}x,
  \qquad
  N_{i,j}(x)=\sum_{\nu=0}^{p^{j}-1}u_{i,j}^{\nu}x = S_{i,j}x,
\]
where $S_{i,j}:=\sum_{\nu=0}^{p^{j}-1}u_{i,j}^{\nu}\in\mathbb{Z}$. We have
\[
  \Ker(N_{i,j})
  =\{x\in\mathbb{Z}/p^{r}\mathbb{Z}\mid S_{i,j}x\equiv 0\pmod{p^{r}}\}.
\]
Writing $S_{i,j}=p^{\alpha}w$ with $\alpha\ge 0$ and $\gcd(w,p)=1$, the condition
$S_{i,j}x\equiv 0\pmod{p^{r}}$ holds if and only if $p^{r-\alpha}\mid x$.
Hence
\begin{equation}\label{KerNijeqgene}
  |\Ker(N_{i,j})|
  = p^{\,\min\{\,r,\ v_p(S_{i,j})\,\}}.
\end{equation}

Assume that $p$ is odd.
From \eqref{eq:well-definedness}, we have $k^{p^{n}}\equiv 1\pmod p$.
Since $(\mathbb{Z}/p\mathbb{Z})^{\times}$ has order $p-1$ and
$\gcd(p^{n},p-1)=1$, it follows that $k\equiv 1\pmod p$.
Together with \eqref{eq:Tijformula}, this gives $u_{i,j}\equiv 1\pmod p$.
If $u_{i,j}=1$, then $S_{i,j}=p^{j}$ and $v_p(S_{i,j})=j$.
If $u_{i,j}\neq 1$, applying Lemma~\ref{lem:LTElemma}, \eqref{eq:LTE-podd}
to $x=u_{i,j}$, $y=1$, and $n=p^{j}$ yields
\[
  v_p(u_{i,j}^{p^{j}}-1)=v_p(u_{i,j}-1)+v_p(p^{j})
  =v_p(u_{i,j}-1)+j.
\]
Since $u_{i,j}^{p^{j}}-1=(u_{i,j}-1)S_{i,j}$, we obtain $v_p(S_{i,j})=j$.
Therefore \eqref{KerNijeqgene} gives
\[
  |\Ker(N_{i,j})|
  =p^{\,\min\{\,r,\ j\,\}}
  =p^{\,\min\{\,m-i,\ j\,\}}.
\]

Assume that $p=2$.
From \eqref{eq:well-definedness} we have $k^{2^{n}}\equiv 1\pmod 2$, so $k$ is odd.
Hence $u_{i,j}\equiv 1\pmod 2$ for all $i,j$.
If $u_{i,j}\equiv 1\pmod{2^{r}}$, then $T_{i,j}=\id_{M_i}$, so
$N_{i,j}=2^{j}\,\id_{M_i}$ and \eqref{KerNijeqgene} yields
$|\Ker(N_{i,j})|=2^{\min\{r,j\}}$.
If $u_{i,j}\not\equiv 1\pmod{2^{r}}$, then $u_{i,j}\neq 1$ and $j\ge 1$.
Applying Lemma~\ref{lem:LTElemma}, \eqref{eq:LTE-peven-neven}
to $x=u_{i,j}$, $y=1$, and $n=2^{j}$ gives
\[
  v_2(u_{i,j}^{2^{j}}-1)
  =v_2(u_{i,j}-1)+v_2(u_{i,j}+1)+v_2(2^{j})-1
  =v_2(u_{i,j}-1)+v_2(u_{i,j}+1)+j-1.
\]
Since $u_{i,j}^{2^{j}}-1=(u_{i,j}-1)S_{i,j}$, we obtain
$v_2(S_{i,j})=v_2(u_{i,j}+1)+j-1$.
Therefore \eqref{KerNijeqgene} gives the second formula in the statement.
\end{proof}

\begin{remark}\label{remark:podd-anotherproof}
By Corollary~\ref{cor:aptG-KerNij-eq} and
Proposition~\ref{prop:KerNij-podd-2-case}, when $p$ is an odd
prime, we obtain
\[
a_{p^{t}}(G)
=
\sum_{i=\max\{0,\,t-n\}}^{\min\{t,\,m\}}
  p^{\,\min\{\,m-i,\ t-i\,\}}.\]
In particular, $a_{p^{t}}(G)$ does not depend on $k$.
Therefore, for any integers $k_{1},k_{2}$ satisfying \eqref{eq:well-definedness},
we have $\zeta_{G(p,m,n,k_{1})}(s)=\zeta_{G(p,m,n,k_{2})}(s)$.
This agrees with Corollary~\ref{cor:pge3-all-similar}.
\end{remark}

We keep the notation $G:=G(2,m,n,k)$.
For each integer $t$, we have
\begin{equation}\label{eq:a2t-by-Eij}
a_{2^{t}}(G)
=
\sum_{i=\max\{0,\,t-n\}}^{\min\{t,\,m\}}
  2^{\,E_{i,\,t-i}(k)}.
\end{equation}

Using \eqref{eq:Tijformula}, we keep in mind that $u_{i,j}\equiv k^{\,2^{\,n-j}}\pmod{2^{\,m-i}}$. Define $E_{i,j}(k)$ by
\begin{subequations}\label{eq:Eij-cases}
\begin{empheq}[left={E_{i,j}(k)=\empheqlbrace}]{align}
  &\min\{\,m-i,\,j\,\}
  &&\textup{if } u_{i,j}\equiv 1\pmod{2^{\,m-i}},
  \label{eq:Eij-case1}\\
  &\min\{\,m-i,\,v_{2}(u_{i,j}+1)+j-1\,\}
  &&\textup{if } u_{i,j}\not\equiv 1\pmod{2^{\,m-i}}.
  \label{eq:Eij-case2}
\end{empheq}
\end{subequations}

\begin{lem}\label{lem:v2kplus-implies-v2kminus}
Let $k$ be an odd integer.
If $v_{2}(k+1)\ge 2$, then $v_{2}(k-1)=1$.
In particular, for any integer $m\ge 1$ we have $\min\{m,\,v_{2}(k-1)\}=1$.
\end{lem}

\begin{proof}
Assume that $v_{2}(k+1)\ge 2$.
Then $4\mid (k+1)$, so $k\equiv -1 \pmod 4$.
Hence $k-1\equiv  2 \pmod 4$.
Therefore $2\mid (k-1)$ and $4\nmid (k-1)$, which means $v_{2}(k-1)=1$.
The last assertion follows immediately.
\end{proof}

\begin{thm}\label{thm:metap=2主結果}
Let $m,n\ge 1$ and let $k_{1},k_{2}$ be integers satisfying \eqref{eq:well-definedness}.
\begin{enumerate}
\item[\textup{(1)}] Assume $n<m$.
Then $\zeta_{G(2,m,n,k_{1})}(s)=\zeta_{G(2,m,n,k_{2})}(s)$
if and only if one of the following conditions holds:
\[
\begin{cases}
\text{\textup{(A)}}\ \ v_2(k_1+1)=v_2(k_2+1)\le m-n,\\
\text{\textup{(B)}}\ \ \min\{v_2(k_1+1), v_2(k_2+1)\}> m-n.
\end{cases}
\]

\item[\textup{(2)}] Assume $n\ge m$.
Then, $\zeta_{G(2,m,n,k_{1})}(s)=\zeta_{G(2,m,n,k_{2})}(s)$.
\end{enumerate}
\end{thm}

\begin{proof}
Assume $p=2$.
For each integer $k$ satisfying \eqref{eq:well-definedness}, set $G(k):=G(2,m,n,k)$.
For each integer $t\ge 0$, write $a_{2^{t}}(k):=a_{2^{t}}\bigl(G(k)\bigr)$.

Fix an integer $t$.
Let $i$ satisfy $\max\{0,\,t-n\}\le i\le \min\{t,\,m\}$, and set $j:=t-i$ and $r:=m-i$.
By \eqref{eq:a2t-by-Eij}, we have
\[
  a_{2^{t}}(k)
  =
  \sum_{i=\max\{0,\,t-n\}}^{\min\{t,\,m\}}
    2^{\,E_{i,\,t-i}(k)}.
\]
Here $E_{i,j}(k)$ is defined by \eqref{eq:Eij-cases}.
Moreover, \eqref{eq:Tijformula} gives the congruence $
  u_{i,j}\equiv k^{\,2^{\,n-j}}\pmod{2^{\,m-i}}$.
We always take $u_{i,j}$ to be the unique integer with
$0\le u_{i,j}<2^{m-i}$ satisfying this congruence.
In particular, for any integer $x$, if $u_{i,j}\equiv x \pmod{2^{m-i}}$, then
\begin{equation}\label{eq:v2-reduction-uij}
  v_2(u_{i,j}+1)=\min\{m-i,\,v_2(x+1)\}.
\end{equation}

\medskip
\noindent
\textbf{Case 1. $\bm{0\le j\le n-1}$.}
We show that, in this range, the values $E_{i,j}(k)$ are independent of $k$.

\medskip\noindent
\textbf{Case 1-1. $\bm{j=0}$.}
Here $B_{0}\simeq\{\id\}$.
By \eqref{eq:Z1-as-kernel} we have $\Ker(N_{i,0})
  \simeq Z^{1}(B_{0},H/A_{i})
  \simeq\{0\}$, 
hence $|\Ker(N_{i,0})|=1$ and therefore $E_{i,0}(k)=0$.
This does not depend on $k$.

\medskip\noindent
\textbf{Case 1-2. $\bm{1\le j\le n-1}$.}
Here $n-j\ge 1$.
Since $k$ satisfies \eqref{eq:well-definedness}, we have $k^{2^{n}}\equiv 1\pmod{2^{m}}$.
Hence $k$ is odd.
For any odd integer $\alpha$ and any integer $\beta\ge 1$, we have
$\alpha^{2^{\beta}}\equiv 1\pmod 8$, hence $k^{2^{\,n-j}}\equiv 1\pmod 8$.
Recalling that $u_{i,j}\equiv k^{2^{\,n-j}}\pmod{2^{r}}$, we distinguish two subcases.

\medskip\noindent
\textbf{Case 1-2-1. $\bm{2\le r}$.}
Then $u_{i,j}\equiv k^{2^{\,n-j}}\equiv 1\pmod 4$, so $v_{2}(u_{i,j}+1)=1$.
Using \eqref{eq:Eij-case2}, we obtain
\[
  E_{i,j}(k)
  =
  \min\{m-i,\,v_{2}(u_{i,j}+1)+j-1\}
  =
  \min\{m-i,\,1+(j-1)\}
  =
  \min\{m-i,\,j\}.
\]
Thus $E_{i,j}(k)=\min\{m-i,\,j\}$, and this does not depend on $k$.

\medskip\noindent
\textbf{Case 1-2-2. $\bm{r=1}$.}
Here $2^{m-i}=2$ and $u_{i,j}\in\{0,1\}$.
Since $k^{2^{\,n-j}}$ is odd, we have $k^{2^{\,n-j}}\equiv 1\pmod 2$.
Therefore $u_{i,j}\equiv 1\pmod 2$, hence $u_{i,j}=1$.
In particular, $u_{i,j}\equiv 1\pmod{2^{m-i}}$, so only \eqref{eq:Eij-case1} can occur.
Hence $E_{i,j}(k)=\min\{1,\,j\}=1$, which again is independent of $k$.

Combining Case 1-1 and Case 1-2, we conclude that, in Case 1, $E_{i,j}(k)=\min\{m-i,\,j\}$, 
and this is independent of $k$.

\medskip\noindent
\medskip\noindent
\textbf{Case 2. $\bm{j=n}$.}
Substituting $n-j=0$ into \eqref{eq:Tijformula}, we obtain
$u_{i,n}\equiv k\pmod{2^{\,m-i}}$.
Therefore
\[
  u_{i,n}\equiv 1\pmod{2^{\,m-i}}
  \iff
  k\equiv 1\pmod{2^{\,m-i}}
  \iff
  v_{2}(k-1)\ge m-i.
\]
Moreover, since $u_{i,n}\equiv k \pmod{2^{\,m-i}}$, we have $
  v_{2}(u_{i,n}+1)=\min\{m-i,\,v_{2}(k+1)\}$
 by \eqref{eq:v2-reduction-uij}. Using these, \eqref{eq:Eij-case1} and \eqref{eq:Eij-case2} can be rewritten as
\begin{subequations}\label{eq:Eij-cases2}
  \begin{empheq}[left={E_{i}(k):=E_{i,n}(k)=\empheqlbrace}]{align}
    &\min\{m-i,\,n\}
      &&\text{if } v_{2}(k-1)\ge m-i,
      \label{eq:Eij-case1-ex}\\
    &\min\{m-i,\,v_{2}(k+1)+n-1\}
      &&\text{if } v_{2}(k-1)<m-i.
      \label{eq:Eij-case2-ex}
  \end{empheq}
\end{subequations}

Since we have been using $i+j=t$, in the present case we have $t=n+i$ with $0\le i\le m-1$.
For each $0\le i\le m-1$ we can write
\[
  a_{2^{\,n+i}}(k)=C_{n+i}+2^{E_{i}(k)},
\]
where we define
\[
  C_{n+i}
  :=
  \sum_{j=\max\{0,\,n+i-m\}}^{n-1} 2^{E_{n+i-j,\,j}(k)}
  =
  \sum_{j=\max\{0,\,n+i-m\}}^{n-1} 2^{\min\{\,m-(n+i-j),\,j\,\}}.
\]
By Case~1, each term $E_{n+i-j,\,j}(k)$ with $0\le j\le n-1$ is independent of $k$.
Hence $C_{n+i}$ does not depend on $k$.
Since Case~1 also shows that $a_{2^{t}}(k)$ is independent of $k$ for $0\le t\le n-1$,
the only possible dependence of $\zeta_{G(k)}(s)$ on $k$ comes from the terms $a_{2^{n+i}}(k)$ with $0\le i\le m-1$.
Therefore we obtain the equivalence
\begin{equation}\label{eq:ゼータ関数一致のiff}
  \zeta_{G(k_{1})}(s)=\zeta_{G(k_{2})}(s)
  \iff
  E_{i}(k_{1})=E_{i}(k_{2})
  \quad(0\le i\le m-1).
\end{equation}

For each integer $k$ satisfying \eqref{eq:well-definedness}, set
\[
  s_{2}(k):=v_{2}(k+1),\qquad
  c'(k):=\min\{m,\,v_{2}(k-1)\},\qquad
  \sigma(k):=s_{2}(k)+n-1.
\]
In terms of this notation, \eqref{eq:Eij-case1-ex} and \eqref{eq:Eij-case2-ex} become
\begin{subequations}\label{eq:Eij-cases3}
  \begin{empheq}[left={E_{i}(k)=\empheqlbrace}]{align}
    &\min\{m-i,\,n\}
      &&\text{if } i\ge m-c'(k),
      \label{eq:Eij-case1-ex2}\\
    &\min\{m-i,\,\sigma(k)\}
      &&\text{if } i<m-c'(k).
      \label{eq:Eij-case2-ex2}
  \end{empheq}
\end{subequations}

\medskip\noindent
\textbf{Case 2-1. $\bm{n<m}$.}
We prove statement \textup{(1)} of the theorem.

\medskip\noindent
\textbf{Case 2-1-A. $\bm{s_{2}(k)\le m-n}$.}
We distinguish two subcases.

\medskip\noindent
\textbf{Case 2-1-A1. $\bm{s_{2}(k)=1}$.}
Then $\sigma(k)=n$.
Hence \eqref{eq:Eij-case1-ex2} and \eqref{eq:Eij-case2-ex2} give
$E_{i}(k)=\min\{m-i,\,n\}$ for every $0\le i\le m-1$.
Thus, in this subcase, the sequence $(E_i(k))$ does not depend on $k$.
In particular, if $s_{2}(k_{1})=s_{2}(k_{2})=1$, then $E_{i}(k_{1})=E_{i}(k_{2})$ for all $i$.
By \eqref{eq:ゼータ関数一致のiff}, we obtain
$\zeta_{G(k_{1})}(s)=\zeta_{G(k_{2})}(s)$.

\medskip\noindent
\textbf{Case 2-1-A2. $\bm{2\le s_{2}(k)\le m-n}$.}
Since $k$ is odd and $s_{2}(k)\ge 2$, Lemma~\ref{lem:v2kplus-implies-v2kminus} gives $v_{2}(k-1)=1$.
Hence $c'(k)=1$.
Therefore $m-c'(k)=m-1$.
For $0\le i\le m-2$, we have $i<m-c'(k)$, so \eqref{eq:Eij-case2-ex2} gives
\[
  E_{i}(k)=\min\{m-i,\,\sigma(k)\}.
\]
Moreover, $2\le s_{2}(k)\le m-n$ implies
\[
  n+1\le \sigma(k)=s_{2}(k)+n-1\le m-1.
\]
In particular, $\sigma(k)\le m-1$, so $E_{0}(k)=\sigma(k)$.

We first prove sufficiency in this subcase.
If $k_{1}$ and $k_{2}$ satisfy $s_{2}(k_{1})=s_{2}(k_{2})$ and $2\le s_{2}(k_{1})\le m-n$,
then $\sigma(k_{1})=\sigma(k_{2})$ and $c'(k_{1})=c'(k_{2})=1$.
Hence \eqref{eq:Eij-cases3} gives $E_{i}(k_{1})=E_{i}(k_{2})$ for every $i$.
By \eqref{eq:ゼータ関数一致のiff}, we obtain
$\zeta_{G(k_{1})}(s)=\zeta_{G(k_{2})}(s)$.

Next we prove necessity in Case 2-1-A.
Assume that $E_{i}(k_{1})=E_{i}(k_{2})$ holds for all $0\le i\le m-1$ and that
$s_{2}(k_{1})\le m-n$ and $s_{2}(k_{2})\le m-n$.
If $s_{2}(k_{\nu})=1$ for some $\nu\in\{1,2\}$, then $E_{0}(k_{\nu})=n$ by Case 2-1-A1.
If $s_{2}(k_{\mu})\ge 2$ for some $\mu\in\{1,2\}$, then Case 2-1-A2 gives
$E_{0}(k_{\mu})=\sigma(k_{\mu})\ge n+1$.
Therefore $E_{0}(k_{1})=E_{0}(k_{2})$ implies that either
$s_{2}(k_{1})=s_{2}(k_{2})=1$ or $2\le s_{2}(k_{1}),s_{2}(k_{2})\le m-n$.

Assume $2\le s_{2}(k_{1}),s_{2}(k_{2})\le m-n$.
Then Case 2-1-A2 gives $E_{0}(k_{\nu})=\sigma(k_{\nu})$ for $\nu\in\{1,2\}$.
Since $E_{0}(k_{1})=E_{0}(k_{2})$, it follows that $\sigma(k_{1})=\sigma(k_{2})$.
Hence $s_{2}(k_{1})=s_{2}(k_{2})$.
Combining this with Case 2-1-A1, we conclude that, in Case 2-1-A,
\[
  E_{i}(k_{1})=E_{i}(k_{2})\ (0\le i\le m-1)
  \iff
  s_{2}(k_{1})=s_{2}(k_{2})\le m-n.
\]
By \eqref{eq:ゼータ関数一致のiff}, this is equivalent to
\[
  \zeta_{G(k_{1})}(s)=\zeta_{G(k_{2})}(s)
  \iff
  v_{2}(k_{1}+1)=v_{2}(k_{2}+1)\le m-n.
\]

\medskip\noindent
\textbf{Case 2-1-B. $\bm{s_{2}(k)\ge m-n+1}$.}
Assume $s_{2}(k)\ge m-n+1$.
Since $n<m$, we have $m-n+1\ge 2$.
Therefore Lemma~\ref{lem:v2kplus-implies-v2kminus} gives $c'(k)=1$.
Moreover, $\sigma(k)=s_{2}(k)+n-1\ge m$.
Hence, for every $0\le i\le m-1$, \eqref{eq:Eij-case2-ex2} yields
\[
  E_{i}(k)=\min\{m-i,\,\sigma(k)\}=m-i.
\]
Thus the sequence $(E_{i}(k))_{0\le i\le m-1}$ is independent of $k$ in this subcase.
In particular, if $s_{2}(k_{1})\ge m-n+1$ and $s_{2}(k_{2})\ge m-n+1$, then
$E_{i}(k_{1})=E_{i}(k_{2})$ for all $i$, so
$\zeta_{G(k_{1})}(s)=\zeta_{G(k_{2})}(s)$ by \eqref{eq:ゼータ関数一致のiff}.

Conversely, assume that $E_{i}(k_{1})=E_{i}(k_{2})$ holds for all $0\le i\le m-1$
and that $s_{2}(k_{1})\ge m-n+1$.
Then $E_{0}(k_{1})=m$ in this subcase.
Hence $E_{0}(k_{2})=m$.
Since $E_{0}(k_{2})=\min\{m,\,\sigma(k_{2})\}$ by \eqref{eq:Eij-cases3},
we obtain $\sigma(k_{2})\ge m$, so $s_{2}(k_{2})\ge m-n+1$.
Therefore, in Case 2-1-B,
\[
  \zeta_{G(k_{1})}(s)=\zeta_{G(k_{2})}(s)
  \iff
  v_{2}(k_{1}+1)\ge m-n+1 \text{ and } v_{2}(k_{2}+1)\ge m-n+1.
\]

Combining Case 2-1-A and Case 2-1-B, we obtain statement \textup{(1)} of the theorem.

\medskip\noindent
\textbf{Case 2-2. $\bm{n\ge m}$.}
For any $0\le i\le m-1$, we have $\min\{m-i,\,n\}=m-i$.
Since $k$ is odd, we have $s_{2}(k)\ge 1$, so $\sigma(k)=s_{2}(k)+n-1\ge m$.
Therefore \eqref{eq:Eij-cases3} gives $E_{i}(k)=m-i$ for every $0\le i\le m-1$.
Hence $E_{i}(k_{1})=E_{i}(k_{2})$ holds for all $i$ and for all $k_{1},k_{2}$ satisfying \eqref{eq:well-definedness}.
By \eqref{eq:ゼータ関数一致のiff}, this implies
$\zeta_{G(k_{1})}(s)=\zeta_{G(k_{2})}(s)$.
This proves statement \textup{(2)} of the theorem.
\end{proof}

\begin{exam}\label{exam:pmn253case}
We consider the case $(p,m,n)=(2,5,3)$ and apply
Theorem~\ref{thm:metap=2主結果}.
By \eqref{eq:well-definedness}, the group $G(2,5,3,k)$ is valid if and only if
$k^{8}\equiv 1 \pmod{32}$.
Any such $k$ must be odd.
Conversely, since the multiplicative group
$(\mathbb{Z}/32\mathbb{Z})^{\times}\simeq \mathbb{Z}_{2}\times \mathbb{Z}_{8}$
has exponent $8$, every odd integer $k$ satisfies $k^{8}\equiv 1 \pmod{32}$.
Hence, the admissible parameters $k$ are precisely the odd residues modulo $32$, namely
\begin{equation}\label{eq:exampmn253-welldefind}
  k \in \{1,3,5,7,9,11,13,15,17,19,21,23,25,27,29,31\}.
\end{equation}

Moreover, by \eqref{eq:同型の条件メタサイクル}, we have
\[
  G(2,5,3,k_{1})\simeq G(2,5,3,k_{2})
  \iff
  \exists\, v\in(\mathbb{Z}/8\mathbb{Z})^{\times}
  \text{ such that } k_{2}\equiv k_{1}^{v}\pmod{32}.
\]
Among the values of $k$ in \eqref{eq:exampmn253-welldefind}, the isomorphism
classes are given by the following partition:
\begin{equation}\label{eq:exampmn253-iso}
  \{1\},\ \{3,11,19,27\},\ \{5,13,21,29\},\ \{7,23\},\ \{9,25\},\ \{15\},\ \{17\},\ \{31\}.
\end{equation}
Thus we may take a complete set of representatives $k\in\{1,3,5,7,9,15,17,31\}$.
For these values of $k$, we have
\[
v_{2}(k+1)=1 \textup{ for } k\in\{1,5,9,17\},\qquad
v_{2}(k+1)=2 \textup{ for } k=3,\qquad
v_{2}(k+1)\ge 3 \textup{ for } k\in\{7,15,31\}.
\]
Theorem~\ref{thm:metap=2主結果} shows that the group zeta function
of $G(2,5,3,k)$ is determined by $v_{2}(k+1)$ and the threshold $m-n+1=3$.
Therefore these representatives fall into the following three zeta-classes:
\begin{equation}\label{eq:非同型だがゼータ一致クラス}
 \{1,5,9,17\},\qquad \{3\},\qquad \{7,15,31\}. 
\end{equation}
Let $k_{1}$ and $k_{2}$ be distinct elements within the same class.
Then $G(2,5,3,k_{1})$ and $G(2,5,3,k_{2})$ are non-isomorphic.
Theorem~\ref{thm:metap=2主結果} implies that they have the same group zeta function.

\end{exam}

\begin{fac}[\cite{Hironaka2017}]\label{fac:coprime-product-multiplicativity}
Let $G$ and $G'$ be finite $p$-groups with $\zeta_{G}(s)=\zeta_{G'}(s)$.
Let $P'$ be a finite group such that $p\nmid |P'|$.
Then $|G|$ and $|P'|$ are coprime, and hence $\zeta_{G\times P'}(s)=\zeta_{G}(s)\,\zeta_{P'}(s)$.
In particular, we have $\zeta_{G\times P'}(s)=\zeta_{G'\times P'}(s)$. 
Moreover, if $G\not\simeq G'$, then $G\times P'\not\simeq G'\times P'$.
\end{fac}

\begin{exam}\label{exam:pmn253-product-extension}
We continue with the case $(p,m,n)=(2,5,3)$ in Example~\ref{exam:pmn253case}.
By \eqref{eq:exampmn253-iso}, the four groups $G(2,5,3,k)$ with
$k\in\{1,5,9,17\}$ are pairwise non-isomorphic, but have the same group zeta function.

Let $P'$ be a finite group such that $2\nmid |P'|$.
Then, by Fact~\ref{fac:coprime-product-multiplicativity}, the four direct products
\[
  G(2,5,3,k)\times P'\qquad k\in\{1,5,9,17\}
\]
are pairwise non-isomorphic and have the same group zeta function.
\end{exam}

Let $L(G)$ denote the subgroup lattice of a finite group $G$.
We write $L(G_{1})\simeq L(G_{2})$ to mean that the lattices $L(G_{1})$ and $L(G_{2})$
are isomorphic as lattices.
For subgroups $H\subset K$ of $G$, we write $H\prec K$ to indicate that $K$ covers $H$ in $L(G)$.
Then the following holds.

\begin{lem}\label{lem:p-group-lattice-iso-implies-zeta}
Let $p$ be a prime and let $G_{1}$ and $G_{2}$ be finite $p$-groups.
Assume that $|G_{1}|=|G_{2}|$.
Then the following implication holds.
\[
  L(G_{1})\simeq L(G_{2}) \Longrightarrow \zeta_{G_{1}}(s) = \zeta_{G_{2}}(s).
\]
\end{lem}

\begin{proof}
Let $G$ be a finite $p$-group, and let $H\subset K\subset G$.
Then $H\prec K$ in $L(G)$ is equivalent to saying that $H$ is a maximal subgroup of $K$.
Since every maximal subgroup of a finite $p$-group has index $p$, it follows that
$H\prec K$ if and only if $[K:H]=p$.

Let $G_{1}$ and $G_{2}$ be finite $p$-groups with $|G_{1}|=|G_{2}|$,
and let $\Phi\colon L(G_{1})\to L(G_{2})$ be a lattice isomorphism.
Let $A\subset G_{1}$, and take a maximal chain $\{\id\}=A_{0}\prec A_{1}\prec \dotsb \prec A_{r}=A$
in the interval $[\{\id\},A]$.
Then $[A_{i}:A_{i-1}]=p$ for each $i$, hence $|A|=p^{r}$ and $\log_{p}|A|=r$.
Since $\Phi$ preserves cover relations, the chain
$\{\id\}=\Phi(A_{0})\prec \Phi(A_{1})\prec \dotsb \prec \Phi(A_{r})=\Phi(A)$
is a maximal chain of the same length.
Therefore $|\Phi(A)|=|A|$.

It follows that $\Phi$ induces, for each $r\ge 0$, a bijection between the subgroups of $G_{1}$
of order $p^{r}$ and the subgroups of $G_{2}$ of order $p^{r}$.
Hence $a_{p^{r}}(G_{1})=a_{p^{r}}(G_{2})$ for every $r$, and therefore
$\zeta_{G_{1}}(s)=\zeta_{G_{2}}(s)$.
\end{proof}

\begin{remark}
For the groups $G(2,5,3,k_{i})$ in Example~\ref{exam:pmn253case}, we verified in GAP
that, among the non-isomorphic groups $G(2,5,3,k_{i})$, the subgroup lattices
fall into the following lattice-isomorphism classes:
\begin{equation}\label{eq:非同型だが束同型一致クラス}
 \{1,5,9,17\},\ \{3\},\ \{7\},\ \{15\},\ \{31\}.
\end{equation}
\end{remark}

\begin{cor}
In general, the converse of Lemma~\ref{lem:p-group-lattice-iso-implies-zeta} does not hold.
\end{cor}

\begin{proof}
Compare \eqref{eq:非同型だがゼータ一致クラス} with \eqref{eq:非同型だが束同型一致クラス}.
If $k_{1}$ and $k_{2}$ are distinct elements of $\{7,15,31\}$, then
$\zeta_{G(2,5,3,k_{1})}(s)=\zeta_{G(2,5,3,k_{2})}(s)$,
while $L\bigl(G(2,5,3,k_{1})\bigr)\not\simeq L\bigl(G(2,5,3,k_{2})\bigr)$.
This yields a counterexample to the converse implication.
\end{proof}

\begin{remark}
The following questions remain open.
\begin{itemize}
  \item Fix integers $m,n\ge 1$.
  For which pairs $(k_{1},k_{2})$ do the groups $G(p,m,n,k_{1})$ and $G(p,m,n,k_{2})$
  have the same group zeta function, while their subgroup lattices are non-isomorphic?
  \item More generally, let $G_{1}$ and $G_{2}$ be finite $p$-groups with
  $\zeta_{G_{1}}(s)=\zeta_{G_{2}}(s)$ and $G_{1}\not\simeq G_{2}$.
  Under what additional conditions does it follow that $L(G_{1})\simeq L(G_{2})$?
\end{itemize}
\end{remark}

\section*{Acknowledgements}
We are deeply indebted to Professor Koichi Betsumiya for his constant guidance
and supervision throughout this research. We also thank Takara Kondo of Kumamoto University for bringing to our attention the problem of group zeta functions and for helpful discussions. Their support and encouragement were indispensable to the completion of this work.

\end{document}